\documentclass[12pt,oneside,reqno]{amsart}

\usepackage{amsmath,amssymb,amsthm,textcomp}
\usepackage{amsfonts,graphicx}
\usepackage[mathscr]{eucal}
\usepackage{color}
\usepackage{csquotes}
\usepackage[backend=bibtex,%
firstinits=true,%
doi=false,%
isbn=true,%
url=false,%
maxnames=99]{biblatex}%

\AtEveryBibitem{\clearfield{issn}}
\AtEveryCitekey{\clearfield{issn}}
\numberwithin{equation}{section}
\DeclareNameAlias{sortname}{last-first}
\theoremstyle{definition}

\numberwithin{equation}{section}

\newcommand{\ncom}{\newcommand}

\ncom{\beq}{\begin{equation}}
\ncom{\eeq}{\end{equation}}
\ncom{\bea}{\begin{eqnarray*}}
\ncom{\eea}{\end{eqnarray*}}
\ncom{\beqa}{\begin{eqnarray}}
\ncom{\eeqa}{\end{eqnarray}}
\ncom{\nno}{\nonumber}
\ncom{\non}{\nonumber}
\ncom{\ds}{\displaystyle}
\ncom{\half}{\frac{1}{2}}
\ncom{\mbx}{\makebox{.25cm}}
\ncom{\hs}{\mbox{\hspace{.25cm}}}
\ncom{\rar}{\rightarrow}
\ncom{\Rar}{\Rightarrow}
\ncom{\noin}{\noindent}
\ncom{\bc}{\begin{center}}
\ncom{\ec}{\end{center}}
\ncom{\sz}{\scriptsize}
\ncom{\rf}{\ref}
\ncom{\s}{\sqrt{2}}
\ncom{\sgm}{\sigma}
\ncom{\Sgm}{\Sigma}
\ncom{\psgm}{\sigma^{\prime}}
\ncom{\dt}{\delta}
\ncom{\Dt}{\Delta}
\ncom{\lmd}{\lambda}
\ncom{\Lmd}{\Lambda}
\ncom{\Th}{\Theta}
\ncom{\e}{\eta}
\ncom{\eps}{\epsilon}
\ncom{\pcc}{\stackrel{P}{>}}
\ncom{\lp}{\stackrel{L_{p}}{>}}
\ncom{\dist}{{\rm\,dist}}
\ncom{\sspan}{{\rm\,span}}
\ncom{\re}{{\rm Re\,}}
\ncom{\im}{{\rm Im\,}}
\ncom{\sgn}{{\rm sgn\,}}
\ncom{\ba}{\begin{array}}
\ncom{\ea}{\end{array}}
\ncom{\hone}{\mbox{\hspace{1em}}}
\ncom{\htwo}{\mbox{\hspace{2em}}}
\ncom{\hthree}{\mbox{\hspace{3em}}}
\ncom{\hfour}{\mbox{\hspace{4em}}}
\ncom{\vone}{\vskip 2ex}
\ncom{\vtwo}{\vskip 4ex}
\ncom{\vonee}{\vskip 1.5ex}
\ncom{\vthree}{\vskip 6ex}
\ncom{\vfour}{\vspace*{8ex}}
\ncom{\norm}{\|\;\;\|}
\ncom{\integ}[4]{\int_{#1}^{#2}\,{#3}\,d{#4}}
\ncom{\vspan}[1]{{{\rm\,span}\{ #1 \}}}
\ncom{\dm}[1]{ {\displaystyle{#1} } }
\ncom{\ri}[1]{{#1} \index{#1}}

\newtheorem{theorem}{\bf Theorem}[section]
\newtheorem{remark}{\bf Remark}[section]

\newtheorem{lemma}{Lemma}[section]
\newtheorem{corollary}{Corollary}[section]
\newtheorem{example}{Example}[section]

\newtheoremstyle
    {remarkstyle}
    {}
    {11pt}
    {}
    {}
    {\bfseries}
    {:}
    {     }
    {\thmname{#1} \thmnumber{#2} }

\theoremstyle{remarkstyle}

\def\eps{\varepsilon}

\begin{document}
\title{Saigo Space-Time Fractional Poisson Process via Adomian Decomposition Method}
\author[Kuldeep Kumar Kataria]{K. K. Kataria}
\address{Kuldeep Kumar Kataria, Department of Mathematics,
 Indian Institute of Technology Bombay, Powai, Mumbai 400076, INDIA.}
 \email{kulkat@math.iitb.ac.in}
\author{P. Vellaisamy}
\address{P. Vellaisamy, Department of Mathematics,
 Indian Institute of Technology Bombay, Powai, Mumbai 400076, INDIA.}
 \email{pv@math.iitb.ac.in}
\thanks{The research of K. K. Kataria was supported by a UGC fellowship no. F.2-2/98(SA-I), Govt. of India.}
\date{February 06, 2017}
\begin{abstract}
We obtain the state probabilities of various fractional versions of the classical homogeneous Poisson process using an alternate and simpler method known as the Adomian decomposition method (ADM). Generally these state probabilities are obtained by evaluating probability generating function using Laplace transform. A generalization of the space and time fractional Poisson process involving the Caputo type Saigo differential operator is introduced and its state probabilities are obtained using ADM.
\end{abstract}

\maketitle
\section{Introduction}
The distribution of the classical homogeneous Poisson process $\{N(t,\lambda)\}_{t\geq0}$ with intensity parameter $\lambda>0$ is given by
\begin{equation}
p(n,t)=\mathrm{Pr}\{N(t,\lambda)=n\}=\frac{e^{-\lambda t}(\lambda t)^n}{n!},\ \ n=0,1,2,\ldots.
\end{equation}
The state probabilities $p(n,t)$, $n\geq0$, of the homogeneous Poisson process solve the following difference-differential equations:
\begin{equation}\label{rynew}
\frac{\mathrm{d}}{\mathrm{dt}}p(n,t)=-\lambda(1-B)p(n,t),\ \ n\geq 0,
\end{equation}
with $p(-1,t)=0$, $t\geq 0$ and subject to the initial conditions $p(0,0)=1$ and $p(n,0)=0$, $n\geq 1$. In the above Kolmogorov equations, $B$ is the backward shift operator acting on the state space, \textit{i.e.} $B(p(n,t))=p(n-1,t)$.

Recently, many authors introduced various fractional generalizations of the homogeneous Poisson process. The time fractional version is obtained by replacing the time derivative in (\ref{rynew}) with the Riemann-Liouville fractional derivative (see {\small{\sc Laskin}} (2003)) or the Caputo fractional derivative (see {\small{\sc Beghin and Orsingher}} (2009)). The time fractional Poisson process (TFPP) $\{N^\alpha(t,\lambda)\}$, $0<\alpha\leq 1$, is defined as the stochastic process whose probability mass function (pmf) $p^\alpha(n,t)=\mathrm{Pr}\{N^\alpha(t,\lambda)=n\}$, satisfies
\begin{equation}\label{r2newq}
\partial_t^\alpha p^\alpha(n,t)=-\lambda(1-B)p^\alpha(n,t),\ \ n\geq 0,
\end{equation}
with $p^\alpha(-1,t)=0$, $t\geq 0$ and the initial conditions $p^\alpha(0,0)=1$ and $p^\alpha(n,0)=0$, $n\geq 1$. Here $\partial_t^\alpha$ denotes the fractional derivative in Caputo sense defined as
\begin{equation}\label{capt}
\partial_t^{\alpha}f(t):=\left\{
	\begin{array}{ll}
	    \frac{1}{\Gamma{(1-\alpha)}}\int^t_{0} (t-s)^{-\alpha}f'(s)\,\mathrm{d}s,\ \ 0<\alpha<1,\\\\
		f'(t),\ \ \alpha=1.
	\end{array}
\right.
\end{equation}
The pmf of the TFPP is given by
\begin{equation}\label{z2}
p^\alpha(n,t)=\frac{(\lambda t^\alpha)^n}{n!}\sum_{k=0}^{\infty‎}\frac{(k+n)!}{k!}\frac{(-\lambda t^\alpha)^k}{\Gamma\left((k+n)\alpha+1\right)},\ \ n\geq0.
\end{equation}
Moreover, we have (see {\small{\sc Meerschaert}} \textit{et al.} (2011)) for $0<\alpha<1$,
\begin{equation}\label{1.1df}
N^\alpha(t,\lambda)\overset{d}{=}N(E_\alpha(t),\lambda),
\end{equation}
where $\overset{d}{=}$ means equal in distribution and $\{E_{\alpha}(t)\}$ is the inverse $\alpha$-stable subordinator independent of $\{N(t,\lambda)\}$.

{\small{\sc Orsingher and Polito}} (2012) introduced a fractional difference operator in the equations governing the state probabilities of the homogeneous Poisson process to obtain a space fractional generalization. The space fractional Poisson process (SFPP) $\{N_\nu(t,\lambda)\}$, $0<\nu\leq 1$, is defined as the stochastic process whose pmf $p_\nu(n,t)=\mathrm{Pr}\{N_\nu(t,\lambda)=n\}$, satisfies
\begin{equation}\label{nnew}
\frac{\mathrm{d}}{\mathrm{d}t} p_\nu(n,t)=-\lambda^\nu(1-B)^\nu p_\nu(n,t),\ \ n\geq 0,
\end{equation}
with initial conditions $p_\nu(0,0)=1$ and $p_\nu(n,0)=0$, $n\geq 1$. Also, $p_\nu(-n,t)=0$, $t\geq 0$, $n\geq 1$. Here, $(1-B)^\nu=\sum_{r=0}^{\infty}\frac{(\nu)_r}{r!}(-1)^rB^r$ is the fractional difference operator and hence, (\ref{nnew}) can be equivalently written as
\begin{equation}\label{z3}
\frac{\mathrm{d}}{\mathrm{d}t} p_\nu(n,t)=-\lambda^\nu\sum_{r=0}^{n}\frac{(\nu)_r}{r!}(-1)^r p_\nu(n-r,t),\ \ n\geq 0,
\end{equation}
where $(\nu)_r=\nu(\nu-1)\ldots(\nu-k+1)$ denotes the falling factorial. The pmf of the SFPP is given by
\begin{equation}\label{z4}
p_\nu(n,t)=\frac{(-1)^n}{n!}\sum_{k=0}^{\infty‎}\frac{(-\lambda^\nu t)^k}{k!}\frac{\Gamma(k\nu+1)}{\Gamma(k\nu+1-n)},\ \ n\geq0.
\end{equation}
A different characterization of the SFPP is obtained in {\small{\sc Orsingher and Polito}} (2012), where the homogeneous Poisson process $\{N(t,\lambda)\}$ is subordinated by an independent $\nu$-stable subordinator $\{D_{\nu}(t)\}$, $0<\nu<1$, {\it i.e.}
\begin{equation}\label{1.1fd}
N_\nu(t,\lambda)\overset{d}{=}N(D_{\nu}(t),\lambda),\ t\geq0.
\end{equation}

A further generalization, namely, the space and time fractional Poisson process (STFPP) (see {\small{\sc Orsingher and Polito}} (2012)) $\{N^\alpha_\nu(t,\lambda)\}$, $0<\alpha,\nu\leq 1$, is the stochastic process with pmf $p^\alpha_{\nu}(n,t)=\mathrm{Pr}\{N^\alpha_\nu(t,\lambda)=n\}$, satisfying
\begin{equation}\label{nnewo}
\partial^\alpha_t p^\alpha_{\nu}(n,t)=-\lambda^\nu(1-B)^\nu p^\alpha_{\nu}(n,t),\ \ n\geq 0,
\end{equation}
with initial conditions $p^\alpha_{\nu}(0,0)=1$ and $p^\alpha_{\nu}(n,0)=0$, $n\geq 1$. Also, $p^\alpha_{\nu}(-n,t)=0$, $t\geq 0$, $n\geq 1$. Equivalently, (\ref{nnewo}) can be written as
\begin{equation*}
\partial^\alpha_t p^\alpha_{\nu}(n,t)=-\lambda^\nu\sum_{r=0}^{n}\frac{(\nu)_r}{r!}(-1)^r p^\alpha_{\nu}(n-r,t),\ \ n\geq 0.
\end{equation*}
The pmf of the STFPP is given by
\begin{equation*}
p^\alpha_{\nu}(n,t)=\frac{(-1)^n}{n!}\sum_{k=0}^{\infty‎}\frac{(-\lambda^\nu t^\alpha)^k}{\Gamma(k\alpha+1)}\frac{\Gamma(k\nu+1)}{\Gamma(k\nu+1-n)},\ \ n\geq0.
\end{equation*}
For $\alpha=\nu=1$ in (\ref{r2newq}), (\ref{nnew}) and (\ref{nnewo}) the TFPP, SFPP and STFPP reduces to classical homogeneous Poisson process. {\small{\sc Polito and Scalas}} (2016) introduced and studied a further generalization of the STFPP which involves the Prabhakar derivative.

The state probabilities of such fractional Poisson processes are generally obtained by evaluating the corresponding probability generating function using Laplace transform. Also, in view of (\ref{1.1df}) and (\ref{1.1fd}) the state probabilities of the SFPP and TFPP can be obtained from the density of the stable and inverse stable subordinator, respectively.

In this note, we obtain these state probabilities by using an alternative method known as the Adomian decomposition method (ADM), which is more direct and much simpler. The method is effective in cases where Laplace transform of a certain fractional derivative is either not known or have complicated form. We also improve the result of {\small{\sc Rao}} {\it et al.} (2010) by introducing the correct version of the Caputo type Saigo fractional derivative. A generalization of the STFPP, namely, the Saigo space and time fractional Poisson process (SSTFPP), which involves the Saigo fractional derivatives in Caputo sense is introduced. As an illustration we obtain the state probabilities of SSTFPP using ADM which are otherwise difficult to obtain using prevalent methods.

\section{Adomian decomposition method}
In ADM (see {\small{\sc Adomian}} (1986), (1994)), solution of the functional equation
\begin{equation}\label{1.1}
u=f+N(u),
\end{equation}
where $N$ is a nonlinear operator and $f$ is a known function, is expressed in the form of an infinite series
\begin{equation}\label{1.2}
u= ‎‎\sum_{n=0}^{\infty‎}‎u_n.
\end{equation}
The nonlinear term $N(u)$ decomposes as 
\begin{equation}\label{1.3}
N(u)=‎‎\sum_{n=0}^{\infty‎}A_n(u_0,u_1,\ldots,u_n),
\end{equation}
where $A_n$ denotes the $n$-th Adomian polynomial in $u_0,u_1,\ldots,u_n$. Also, the series (\ref{1.2}) and (\ref{1.3}) are assumed to be absolutely convergent. So, (\ref{1.1}) can be rewritten as
\begin{equation}\label{1.4}
‎‎\sum_{n=0}^{\infty‎}‎u_n=f+‎‎\sum_{n=0}^{\infty‎}A_n(u_0,u_1,\ldots,u_n).
\end{equation}
Thus $u_n$'s are obtained by the following recursive relation
\begin{equation*}\label{1.5}
u_0=f\ \ \ \ \mathrm{and}\ \ \ \ u_n=A_{n-1}(u_0,u_1,\ldots,u_{n-1}).
\end{equation*}
The crucial step involved in ADM is the calculation of Adomian polynomials. {\small{\sc Adomian} (1986) gave a method for determining these polynomials, by parametrizing $u$ as
\begin{equation*}\label{1.4a}
u_\lambda=‎‎\sum_{n=0}^{\infty‎}u_n\lambda^n 
\end{equation*}
and assuming $N(u_\lambda)$ to be analytic in $\lambda$, which decomposes as
\begin{equation*}\label{1.5a}
N(u_\lambda)=‎‎\sum_{n=0}^{\infty‎}A_n(u_0,u_1,\ldots,u_n)\lambda^n.
\end{equation*}
Hence, Adomian polynomials are given by
\begin{equation}\label{1.6}
A_n(u_0,u_1,\ldots,u_n)=\left.\frac{1}{n!}\frac{\partial^n N(u_\lambda)}{\partial \lambda^n} \right|_{\lambda=0},\ \forall\ n\geq 0.
\end{equation}
{\small{\sc Rach}} (1984) suggested the following formula for these polynomials: $A_0(u_0)=‎‎N(u_0)$,
\begin{equation}\label{1.7}
A_n(u_0,u_1,\ldots,u_n)=\sum_{k=1}^{n‎}C(k,n)N^{(k)}(u_0),\ \forall\ n\in\mathbb{N},
\end{equation}
where 
\begin{equation*}\label{1.8}
C(k,n)=\underset{\sum_{j=1}^nk_j=k\ ,\ k_j\in\mathbb{N}_0}{\sum_{\sum_{j=1}^{n}jk_j=n}}\prod_{j=1}^{n}\frac{u_j^{k_j}}{k_j!},
\end{equation*}
and $N^{(k)}(.)$ denotes the $k$-th derivative of the nonlinear term. One can easily show the equivalence of (\ref{1.6}) and (\ref{1.7}) using the Fa\`{a} di Bruno's formula. Recently, {\small{\sc Kataria and Vellaisamy}} (2016) obtained simple parametrization methods for generating these Adomian polynomials both explicitly and recursively. For more recent work on Adomian polynomials see {\small{\sc Duan}} (2010), (2011). 

The only crucial and difficult step involved in ADM is the computation of these polynomials. But, for the linear case, $N(u)=u$, $A_n$ simply reduces to $u_n$. Note that the functional equations corresponding to various fractional generalizations of the homogeneous Poisson process does not involve nonlinear term. Hence, ADM conveniently and rapidly gives the state probabilities as the series solutions of the corresponding difference-differential equations.

\section{Application of ADM to fractional Poisson processes}
The state probabilities of certain fractional versions of homogeneous Poisson process are obtained by several authors by evaluating the probability generating functions using Laplace transform, see {\small{\sc Beghin and Orsingher}} (2009), {\small{\sc Polito and Scalas}} (2016) and Remark 3.3 of {\small{\sc Meerschaert}} \textit{et al.} (2011) and references therein. In this section, we apply ADM to obtain the distribution of STFPP. Note that ADM can also be effectively used to independently obtain the state probabilities of TFPP and SFPP (see Supplementary file).

First we define the Riemann-Liouville (RL) fractional integral $I^\alpha_t$ of order $\alpha$,
\begin{equation}
I^\alpha_tf(t):=\frac{1}{\Gamma{(\alpha)}}\int^t_{0} (t-s)^{\alpha-1}f(s)\,\mathrm{d}s,\ \ \alpha>0.
\end{equation}
\begin{remark}
Note that the RL integral is a linear operator. Therefore, the Adomian polynomials $A_k$'s for the case $N(u(t))=cI^\alpha_tu(t)$ are simply $A_k(u_0(t),u_1(t),\ldots,u_k(t))=cI^\alpha_tu_k(t)$, where $c$ is a scalar.
\end{remark}
The following result will be used (see Eq. 2.1.16, Kilbas {\it et. al.} (2006)).
\begin{lemma}\label{lem1}
Let $\alpha,\rho> 0$. Then
\begin{equation*}
I^\alpha_tt^{\rho-1}=\frac{\Gamma(\rho)}{\Gamma(\rho+\alpha)}t^{\rho+\alpha-1}.
\end{equation*}
\end{lemma}
It is known that (see Eq. 2.4.44, Kilbas {\it et. al.} (2006))
\begin{equation}\label{ythw}
I^\alpha_t\partial^\alpha_tf(t)=f(t)-f(0),\ \ 0<\alpha\leq 1,
\end{equation}
where $\partial^\alpha_t$ denotes the Caputo derivative defined in (\ref{capt}).
\subsection{Space and time fractional Poisson process}
The next result is stated without proof in {\small{\sc Orsingher and Polito}} (2012). We give a detailed proof using ADM. 
\begin{theorem}\label{t2kk}
Consider the following difference-differential equations governing the state probabilities of the STFPP:
\begin{equation}\label{hj456}
\partial^\alpha_t p^\alpha_{\nu}(n,t)=-\lambda^\nu\sum_{r=0}^{n}(-1)^r\frac{(\nu)_r}{r!} p^\alpha_{\nu}(n-r,t),\ \ 0<\alpha,\nu\leq 1,\ n\geq 0,
\end{equation}
with $p^\alpha_{\nu}(0,0)=1$ and $p^\alpha_{\nu}(n,0)=0$, $n\geq 1$. The solution of (\ref{hj456}) is given by
\begin{equation}\label{2.4kkt}
p^\alpha_{\nu}(n,t)=\frac{(-1)^n}{n!}\sum_{k=0}^{\infty‎}\frac{(-\lambda^\nu t^\alpha)^k}{\Gamma(k\alpha+1)}\frac{\Gamma(k\nu+1)}{\Gamma(k\nu+1-n)},\ \ n\geq0.
\end{equation}
\end{theorem}
\begin{proof}
Applying RL integral $I^\alpha_t$ on both sides of (\ref{hj456}) and using (\ref{ythw}), we get
\begin{equation}\label{kol}
p^\alpha_{\nu}(n,t)=p^\alpha_{\nu}(n,0)-\lambda^\nu I^\alpha_t \sum_{r=0}^n (-1)^r\frac{(\nu)_r}{r!}p^\alpha_{\nu}(n-r,t),\ \ n\geq 0.
\end{equation}
Note that $p^\alpha_{\nu}(-1,t)=0$ for $t\geq0$. For $n=0$, the above functional equation is of the form (\ref{1.1}), where $N(p^\alpha_{\nu}(0,t))=-\lambda^\nu I_t^\alpha p^\alpha_{\nu}(0,t)$. Therefore, Adomian polynomials $A_k$'s are simply $A_k=-\lambda^\nu I_t^\alpha p^\alpha_{\nu,k}(0,t)$. Substitute $p^\alpha_{\nu}(0,t)=\sum_{k=0}^{\infty}p^\alpha_{\nu,k}(0,t)$ in (\ref{kol}) and apply ADM (see (\ref{1.4})), to get
\begin{equation*}\label{2.7kk}
\sum_{k=0}^{\infty}p^\alpha_{\nu,k}(0,t)=p^\alpha_{\nu}(0,0)-\lambda^\nu\sum_{k=0}^{\infty} I^\alpha_t p^\alpha_{\nu,k}(0,t).
\end{equation*}
Thus, $p^\alpha_{\nu,0}(0,t)=p^\alpha_{\nu}(0,0)=1$ and $p^\alpha_{\nu,k}(0,t)=-\lambda^\nu I^\alpha_t p^\alpha_{\nu,k-1}(0,t)$, $k\geq 1$.\\
Hence,
\begin{align*}
p^\alpha_{\nu,1}(0,t)&=-\lambda^\nu I^\alpha_t p^\alpha_{\nu,0}(0,t)
=-\lambda^\nu I^\alpha_tt^0
=\frac{-\lambda^\nu t^\alpha}{\Gamma(\alpha+1)},\\
p^\alpha_{\nu,2}(0,t)&=-\lambda^\nu I^\alpha_t p^\alpha_{\nu,1}(0,t)
=\frac{\lambda^{2\nu}}{\Gamma(\alpha+1)}I^\alpha_t t^\alpha
=\frac{(-\lambda^\nu t^\alpha)^2}{\Gamma(2\alpha+1)},\\
p^\alpha_{\nu,3}(0,t)&=-\lambda^\nu I^\alpha_t p^\alpha_{\nu,2}(0,t)
=\frac{-\lambda^{3\nu}}{\Gamma(2\alpha+1)}I^\alpha_t t^{2\alpha}
=\frac{(-\lambda^\nu t^\alpha)^3}{\Gamma(3\alpha+1)}.
\end{align*}
Let
\begin{equation}
p^\alpha_{\nu,k-1}(0,t)=\frac{(-\lambda^\nu t^\alpha)^{k-1}}{\Gamma((k-1)\alpha+1)}.
\end{equation}
Then
\begin{equation*}
p^\alpha_{\nu,k}(0,t)=-\lambda^\nu I^\alpha_t p^\alpha_{\nu,k-1}(0,t)
=\frac{(-\lambda^\nu)^k}{\Gamma((k-1)\alpha+1)} I^\alpha_t t^{(k-1)\alpha}
=\frac{(-\lambda^\nu t^\alpha)^k}{\Gamma(k\alpha+1)},\ \ k\geq 0.
\end{equation*}
Therefore
\begin{equation}
p^\alpha_\nu(0,t)=\sum_{k=0}^{\infty}\frac{(-\lambda^\nu t^\alpha)^k}{\Gamma(k\alpha+1)},
\end{equation}
{\it i.e.} the result holds for $n=0$.

For $n=1$, substituting $p^\alpha_{\nu}(1,t)=\sum_{k=0}^{\infty}p^\alpha_{\nu,k}(1,t)$ in (\ref{kol}) and applying ADM, we get
\begin{equation*}\label{2.7kk}
\sum_{k=0}^{\infty}p^\alpha_{\nu,k}(1,t)=p^\alpha_{\nu}(1,0)-\lambda^\nu\sum_{k=0}^{\infty} I^\alpha_t \left(p^\alpha_{\nu,k}(1,t)-\nu p^\alpha_{\nu,k}(0,t)\right).
\end{equation*}
Thus, $p^\alpha_{\nu,0}(1,t)=p^\alpha_{\nu}(1,0)=0$ and $p^\alpha_{\nu,k}(1,t)=-\lambda^\nu I^\alpha_t  \left(p^\alpha_{\nu,k-1}(1,t)-\nu p^\alpha_{\nu,k-1}(0,t)\right)$, $k\geq 1$.\\
Hence,
\begin{align*}
p^\alpha_{\nu,1}(1,t)&=-\lambda^\nu I^\alpha_t  \left(p^\alpha_{\nu,0}(1,t)-\nu p^\alpha_{\nu,0}(0,t)\right)=\frac{-\nu(-\lambda^\nu t^\alpha)}{\Gamma(\alpha+1)},\\
p^\alpha_{\nu,2}(1,t)&=-\lambda^\nu I^\alpha_t  \left(p^\alpha_{\nu,1}(1,t)-\nu p^\alpha_{\nu,1}(0,t)\right)=\frac{-2\nu(-\lambda^{\nu} t^{\alpha})^2}{\Gamma(2\alpha+1)},\\
p^\alpha_{\nu,3}(1,t)&=-\lambda^\nu I^\alpha_t  \left(p^\alpha_{\nu,2}(1,t)-\nu p^\alpha_{\nu,2}(0,t)\right)=\frac{-3\nu(-\lambda^{\nu} t^{\alpha})^3}{\Gamma(3\alpha+1)}.
\end{align*}
Let
\begin{equation}
p^\alpha_{\nu,k-1}(1,t)=\frac{-(k-1)\nu(-\lambda^{\nu} t^{\alpha})^{k-1}}{\Gamma((k-1)\alpha+1)}.
\end{equation}
Then
\begin{align*}
p^\alpha_{\nu,k}(1,t)&=-\lambda^\nu I^\alpha_t  \left(p^\alpha_{\nu,k-1}(1,t)-\nu p^\alpha_{\nu,k-1}(0,t)\right)\\
&=-\lambda^\nu I^\alpha_t  \left(\frac{-(k-1)\nu(-\lambda^{\nu} t^{\alpha})^{k-1}}{\Gamma((k-1)\alpha+1)}-\frac{\nu(-\lambda^{\nu} t^{\alpha})^{k-1}}{\Gamma((k-1)\alpha+1)}\right)\\
&=\frac{-k\nu(-\lambda^{\nu} t^{\alpha})^k}{\Gamma(k\alpha+1)},\ \ k\geq 0.
\end{align*}
Therefore
\begin{equation}
p^\alpha_\nu(1,t)=-\sum_{k=0}^{\infty}\frac{k\nu(-\lambda^{\nu} t^{\alpha})^k}{\Gamma(k\alpha+1)},
\end{equation}
{\it i.e.} the result holds for $n=1$.

For $n=2$, substituting $p^\alpha_{\nu}(2,t)=\sum_{k=0}^{\infty}p^\alpha_{\nu,k}(2,t)$ in (\ref{kol}) and applying ADM, we get
\begin{equation*}
\sum_{k=0}^{\infty}p^\alpha_{\nu,k}(2,t)=p^\alpha_{\nu}(2,0)-\lambda^\nu\sum_{k=0}^{\infty} I^\alpha_t \left(p^\alpha_{\nu,k}(2,t)-\nu p^\alpha_{\nu,k}(1,t)+\frac{\nu(\nu-1)}{2}p^\alpha_{\nu,k}(0,t)\right).
\end{equation*}
Thus, $p^\alpha_{\nu,0}(2,t)=p^\alpha_{\nu}(2,0)=0$ and
\begin{equation*}
p^\alpha_{\nu,k}(2,t)=-\lambda^\nu I^\alpha_t  \left(p^\alpha_{\nu,k-1}(2,t)-\nu p^\alpha_{\nu,k-1}(1,t)+\frac{\nu(\nu-1)}{2}p^\alpha_{\nu,k-1}(0,t)\right),\ \ k\geq 1.
\end{equation*}
Hence,
\begin{align*}
p^\alpha_{\nu,1}(2,t)&=-\lambda^\nu I^\alpha_t  \left(p^\alpha_{\nu,0}(2,t)-\nu p^\alpha_{\nu,0}(1,t)+\frac{\nu(\nu-1)}{2}p^\alpha_{\nu,0}(0,t)\right)=\frac{\nu(\nu-1)(-\lambda^\nu t^\alpha)}{2\Gamma(\alpha+1)},\\
p^\alpha_{\nu,2}(2,t)&=-\lambda^\nu I^\alpha_t  \left(p^\alpha_{\nu,1}(2,t)-\nu p^\alpha_{\nu,1}(1,t)+\frac{\nu(\nu-1)}{2}p^\alpha_{\nu,1}(0,t)\right)=\frac{2\nu(2\nu-1)(-\lambda^\nu t^\alpha)^2}{2\Gamma(2\alpha+1)},\\
p^\alpha_{\nu,3}(2,t)&=-\lambda^\nu I^\alpha_t  \left(p^\alpha_{\nu,2}(2,t)-\nu p^\alpha_{\nu,2}(1,t)+\frac{\nu(\nu-1)}{2}p^\alpha_{\nu,2}(0,t)\right)=\frac{3\nu(3\nu-1)(-\lambda^\nu t^\alpha)^3}{2\Gamma(3\alpha+1)}.
\end{align*}
Let
\begin{equation}
p^\alpha_{\nu,k-1}(2,t)=\frac{(k-1)\nu((k-1)\nu-1)(-\lambda^\nu t^\alpha)^{k-1}}{2\Gamma((k-1)\alpha+1)}.
\end{equation}
Then
\begin{align*}
p^\alpha_{\nu,k}(2,t)&=-\lambda^\nu I^\alpha_t  \left(p^\alpha_{\nu,k-1}(2,t)-\nu p^\alpha_{\nu,k-1}(1,t)+\frac{\nu(\nu-1)}{2}p^\alpha_{\nu,k-1}(0,t)\right)\\
&=\frac{k\nu(k\nu-1)(-\lambda^\nu t^\alpha)^k}{2\Gamma(k\alpha+1)},\ \ k\geq 0.
\end{align*}
Therefore
\begin{equation}
p^\alpha_\nu(2,t)=\frac{1}{2}\sum_{k=0}^{\infty}\frac{k\nu(k\nu-1)(-\lambda^\nu t^\alpha)^k}{\Gamma(k\alpha+1)},
\end{equation}
{\it i.e.} the result holds for $n=2$.

Now assume for $m>2$ the following:
\begin{equation}
p^\alpha_{\nu,k}(m,t)=\frac{(-1)^m}{m!}\frac{(k\nu)_m(-\lambda^\nu t^\alpha)^k}{\Gamma(k\alpha+1)},\ \ k\geq 0,
\end{equation}
{\it i.e.} (\ref{2.4kkt}) holds for $n=m$, where $p^\alpha_{\nu}(m,t)=\sum_{k=0}^{\infty}p^\alpha_{\nu,k}(m,t)$ and $(k\nu)_m$ denotes the falling factorial.

For $n=m+1$, substituting $p^\alpha_{\nu}(m+1,t)=\sum_{k=0}^{\infty}p^\alpha_{\nu,k}(m+1,t)$ in (\ref{kol}) and applying ADM, we get
\begin{equation*}\label{2.7kk}
\sum_{k=0}^{\infty}p^\alpha_{\nu,k}(m+1,t)=p^\alpha_{\nu}(m+1,0)-\lambda^\nu\sum_{k=0}^{\infty} I^\alpha_t \sum_{r=0}^{m+1} (-1)^r\frac{(\nu)_r}{r!}p^\alpha_{\nu,k}(m+1-r,t).
\end{equation*}
Thus, $p^\alpha_{\nu,0}(m+1,t)=p^\alpha_{\nu}(m+1,0)=0$ and 
\begin{equation*}
p^\alpha_{\nu,k}(m+1,t)=-\lambda^\nu I^\alpha_t \sum_{r=0}^{m+1} (-1)^r\frac{(\nu)_r}{r!}p^\alpha_{\nu,k-1}(m+1-r,t),\ \ k\geq 1.
\end{equation*}
Hence,
\begin{align*}
p^\alpha_{\nu,1}(m+1,t)&=-\lambda^\nu I^\alpha_t \sum_{r=0}^{m+1} (-1)^r\frac{(\nu)_r}{r!}p^\alpha_{\nu,0}(m+1-r,t)\\
&=-\lambda^\nu\frac{(-1)^{m+1}}{(m+1)!}(\nu)_{m+1}I^\alpha_tt^0=\frac{(-1)^{m+1}}{(m+1)!}\frac{(\nu)_{m+1}(-\lambda^\nu t^\alpha)}{\Gamma(\alpha+1)},\\
p^\alpha_{\nu,2}(m+1,t)&=-\lambda^\nu I^\alpha_t \sum_{r=0}^{m+1} (-1)^r\frac{(\nu)_r}{r!}p^\alpha_{\nu,1}(m+1-r,t)\\
&=\frac{\lambda^{2\nu}(-1)^{m+1}}{(m+1)!\Gamma(\alpha+1)}I^\alpha_tt^\alpha\sum_{r=0}^{m+1} \frac{(m+1)!}{r!(m+1-r)!}(\nu)_r(\nu)_{m+1-r}\\
&=\frac{(-1)^{m+1}}{(m+1)!}\frac{(2\nu)_{m+1}(-\lambda^\nu t^\alpha)^2}{\Gamma(2\alpha+1)},
\end{align*}
where the last step follows from the binomial theorem for falling factorials. Now let
\begin{equation}
p^\alpha_{\nu,k-1}(m+1,t)=\frac{(-1)^{m+1}}{(m+1)!}\frac{((k-1)\nu)_{m+1}(-\lambda^\nu t^\alpha)^{k-1}}{\Gamma((k-1)\alpha+1)}.
\end{equation}
Then
\begin{align*}
p^\alpha_{\nu,k}(m+1,t)&=-\lambda^\nu I^\alpha_t \sum_{r=0}^{m+1} (-1)^r\frac{(\nu)_r}{r!}p^\alpha_{\nu,k-1}(m+1-r,t)\\
&=\frac{(-\lambda^\nu)^k(-1)^{m+1}I^\alpha_tt^{(k-1)\alpha}}{(m+1)!\Gamma((k-1)\alpha+1)}\sum_{r=0}^{m+1} \frac{(m+1)!}{r!(m+1-r)!}(\nu)_r((k-1)\nu)_{m+1-r}\\
&=\frac{(-1)^{m+1}}{(m+1)!}\frac{(k\nu)_{m+1}(-\lambda^\nu t^\alpha)^k}{\Gamma(k\alpha+1)},\ \ k\geq 0.
\end{align*}
Therefore
\begin{equation*}
p^\alpha_\nu(m+1,t)=\frac{(-1)^{m+1}}{(m+1)!}\sum_{k=0}^{\infty‎}\frac{(-\lambda^\nu t^\alpha)^k}{\Gamma(k\alpha+1)}\frac{\Gamma(k\nu+1)}{\Gamma(k\nu-m)},
\end{equation*}
and thus the result holds for $n=m+1$. This completes the proof.
\end{proof}
\begin{remark}
The state probabilities of TFPP and SFPP can be obtained as special cases of the above result {\textit i.e.} by substituting $\nu=1$ and $\alpha=1$ in Theorem \ref{t2kk}, respectively. However, the difference-differential equations (\ref{r2newq}) and (\ref{z3}) governing the state probabilities of TFPP and SFPP can also be independently solved using ADM to obtain the corresponding distributions.
\end{remark}
\begin{corollary}
Let the random variable $X^\alpha_\nu$ be the waiting time of the first space and time fractional Poisson event. Then the following determine the distribution of $X^\alpha_\nu$:
\begin{equation}\label{77}
\mathrm{Pr}\{X^\alpha_\nu>t\}=\mathrm{Pr}\{N^\alpha_\nu(t,\lambda)=0\}=E_{\alpha}(-\lambda^\nu t^\alpha),\ \ t\geq 0,
\end{equation}
where $E_{\alpha}(.)$ is Mittag-Leffler function defined by 
\begin{equation*}
E_{\alpha}(x)=\sum_{k=0}^{\infty‎}\frac{x^k}{\Gamma(k\alpha+1)},\ \ \alpha>0,\ x\in\mathrm{R}.
\end{equation*}
\end{corollary}
\begin{remark}
The special cases $\alpha=1$ and $\nu=1$ gives the corresponding waiting times of SFPP and TFPP {\it i.e.}
\begin{align*}
\mathrm{Pr}\{X_\nu>t\}&=e^{-\lambda^\nu t},\ \ t\geq0,\\
\mathrm{Pr}\{X^\alpha>t\}&=E_{\alpha}(-\lambda t^\alpha),\ \ t\geq0,
\end{align*}
respectively.
\end{remark}

\section{A generalization of the STFPP}
{\small{\sc Saigo}} (1978) introduced the fractional integral operators with Gauss hypergeometric function as the kernel, which are interesting generalizations of the classical Riemann-Liouville and Erd\'elyi-Kober fractional operators. For real numbers $\alpha>0,\beta$ and $\gamma$, the generalized fractional integral associated with Gauss hypergeometric function is defined by (see {\small{\sc Saigo}} (1978) and {\small{\sc Srivastava}} {\it et al.} (1988)):
\begin{equation}\label{1.8}
I_t^{\alpha,\beta,\gamma}f(t)=\frac{t^{-\alpha-\beta}}{\Gamma(\alpha)}\int_0^t(t-s)^{\alpha-1}{}_{2}F_1\left(\alpha+\beta,-\gamma;\alpha;1-\frac{s}{t}\right)f(s)\,\mathrm{d}s,
\end{equation}
where $f(t)$ is a continuous real valued function on $(0,\infty)$ of order $O(t^\epsilon)$, $\epsilon>\max\{0,\beta-\gamma\}-1$. The Gauss hypergeometric function ${}_{2}F_1(a,b;c;z)$ is defined by
\begin{equation*}\label{1.12}
{}_{2}F_1(a,b;c;z)=\sum_{k=0}^{\infty}\frac{(a)_{k}(b)_{k}}{(c)_{k}}\frac{z^k}{k!},\ \ |z|<1,\ z\in\mathbb{C},
\end{equation*}
where $a,b\in\mathbb{C}$ and $c\in\mathbb{C}\setminus\mathbb{Z}_0^-$. The corresponding fractional differential operator ({\small{\sc Saigo and Maeda}} (1998)) is
\begin{equation}\label{1.11}
D_t^{\alpha,\beta,\gamma}f(t)=\frac{\mathrm{d}^m}{\mathrm{d}t^{m}}I_t^{-\alpha+m,-\beta-m,\alpha+\gamma-m}f(t),
\end{equation}
where $m-1<\alpha\leq m$, $m\in\mathbb{N}$. Substituting $\beta=-\alpha$ ($\beta=0$) in (\ref{1.8}) and (\ref{1.11}), we get the Riemann-Liouville (Erd\'elyi-Kober) integral and differential operator, respectively.

The following is a known result for Saigo fractional integral (see Lemma 3, {\small{\sc Srivastava}} {\it et al.} (1988)).
\begin{lemma}\label{lhg}
Let $\alpha>0,\beta,\gamma$ and $\rho$ be real numbers such that $\rho>\beta-\gamma$. Then
\begin{equation*}\label{2.4}
I_{t}^{\alpha,\beta,\gamma}t^{\rho-1}=\frac{\Gamma(\rho)\Gamma(\rho-\beta+\gamma)}{\Gamma(\rho-\beta)\Gamma(\rho+\alpha+\gamma)}t^{\rho-\beta-1}.
\end{equation*}
\end{lemma}
\noindent For $\beta=-\alpha$, the above result reduces to Lemma \ref{lem1}.
\subsection{Caputo-type modification of Saigo fractional derivative}
 {\small{\sc Rao {\it et al.}}} (2010) introduced the Caputo-type fractional derivative that involves the Gauss hypergeometric function in the kernel. The Caputo fractional differential operator of order $\alpha>0$ associated with the Gauss hypergeometric function is defined by
\begin{equation}\label{1ff}
{}_{*}D_t^{\alpha,\beta,\gamma}f(t)=I_t^{-\alpha+m,-\beta-m,\alpha+\gamma-m}f^{(m)}(t),
\end{equation}
where $m-1<\alpha\leq m$, $m\in\mathbb{N}$ and $f^{(m)}(t)=\frac{\mathrm{d}^m}{\mathrm{d}t^{m}}f(t)$. 

The following semi group property of Saigo integral operator was used to prove Theorem 6 of {\small{\sc Rao {\it et al.}}} (2010):
\begin{equation}\label{3.2lf}
I_t^{\alpha,\beta,\gamma}I_t^{\eta,\delta,\xi}f(t)=I_t^{\eta,\delta,\xi}I_t^{\alpha,\beta,\gamma}f(t).
\end{equation}
We claim that (\ref{3.2lf}) is false and hence Theorem 6 of {\small{\sc Rao {\it et al.}}} (2010) does not hold true for ${}_{*}D_t^{\alpha,\beta,\gamma}$. The counter example follows:
\begin{example}
For $\alpha>0,\eta>0$, $\rho>\max\{\beta-\gamma,\delta-\xi,\beta-\gamma+\delta,\delta-\xi+\beta\}$ and $f(t)=t^{\rho-1}$, it is easy to see using Lemma \ref{lhg} that (\ref{3.2lf}) is contradicted.
\end{example}
Next we introduce a new Caputo version of the Saigo fractional derivative by slight modification of (\ref{1ff}). For real numbers $\alpha>0,\beta$ and $\gamma$, we define a new version of the Caputo fractional differential operator associated with the Gauss hypergeometric function as follows:
\begin{equation}\label{ghj}
\partial_t^{\alpha,\beta,\gamma}f(t)=I_t^{-\alpha+m,-\beta-m,\alpha+\gamma}f^{(m)}(t),
\end{equation}
where $m-1<\alpha\leq m$, $m\in\mathbb{N}$.

Now we show that Theorem 6 of {\small{\sc Rao {\it et al.}}} (2010) holds for new Caputo version of the Saigo fractional derivative $\partial_t^{\alpha,\beta,\gamma}$. The following semi group property of Saigo integral operator (see Eq. (2.22) {\small{\sc Saigo}} (1978)) will be used to prove the next result:
\begin{equation}\label{fghtr}
I_t^{\alpha,\beta,\gamma}I_t^{\eta,\delta,\alpha+\gamma}f(t)=I_t^{\alpha+\eta,\beta+\delta,\gamma}f(t).
\end{equation}
\begin{theorem}
The following composition holds:
\begin{equation}\label{5.236}
I_t^{\alpha,\beta,\gamma}\partial_t^{\alpha,\beta,\gamma}f(t)=f(t)-\sum_{k=0}^{m-1}\frac{f^{(k)}(0)}{k!}t^k,
\end{equation}
where $m-1<\alpha\leq m$.
\end{theorem}
\begin{proof}
Consider
\begin{align*}
I_t^{\alpha,\beta,\gamma}\partial_t^{\alpha,\beta,\gamma}f(t)=I_t^{\alpha,\beta,\gamma}I_t^{-\alpha+m,-\beta-m,\alpha+\gamma}f^{(m)}(t)=I_t^{m,-m,\gamma}f^{(m)}(t)=I_t^{m}f^{(m)}(t),
\end{align*}
and the result follows on using Lemma 2.22 of {\small{\sc Kilbas}} {\it et al.} (2006).
\end{proof}
As a special case we have
\begin{equation}\label{2.478}
I_t^{\alpha,\beta,\gamma}\partial_t^{\alpha,\beta,\gamma}f(t)=f(t)-f(0),\ \ 0<\alpha\leq 1.
\end{equation}
\subsection{Saigo space and time fractional Poisson process}
We define the Saigo space and time fractional Poisson process (SSTFPP) $\{N^{\alpha,\beta,\gamma}_\nu(t,\lambda)\}$ for parameters $0<\alpha,\nu\leq 1$, $\beta<0$ and $\gamma\in\mathbb{R}$ as the stochastic process whose state probabilities $p^{\alpha,\beta,\gamma}_\nu(n,t)=\mathrm{Pr}\{N^{\alpha,\beta,\gamma}_\nu(t,\lambda)=n\}$, satisfies
\begin{equation}\label{rnew}
\partial_{t}^{\alpha,\beta,\gamma} p^{\alpha,\beta,\gamma}_\nu(n,t)=-\lambda^\nu(1-B)^\nu p^{\alpha,\beta,\gamma}_\nu(n,t),\ \ n\geq 0,
\end{equation}
with $p^{\alpha,\beta,\gamma}_\nu(-1,t)=0$ and subject to the initial conditions $p^{\alpha,\beta,\gamma}_\nu(0,0)=1$ and $p^{\alpha,\beta,\gamma}_\nu(n,0)=0$, $n\geq 1$. Also, (\ref{rnew}) can be rewritten as
\begin{equation}\label{rnewq}
\partial_{t}^{\alpha,\beta,\gamma} p^{\alpha,\beta,\gamma}_\nu(n,t)=-\lambda^\nu\sum_{r=0}^n (-1)^r\frac{(\nu)_r}{r!}p^{\alpha,\beta,\gamma}_{\nu}(n-r,t),\ \ n\geq 0.
\end{equation}
For $\beta=-\alpha$, the SSTFPP reduces to STFPP.

\begin{theorem}\label{t2kkt}
The probability mass function, $p^{\alpha,\beta,\gamma}_\nu(n,t)$, of the SSTFPP $\{N^{\alpha,\beta,\gamma}_\nu(t,\lambda)\}$ is
\begin{equation}\label{2.4kkr}
p^{\alpha,\beta,\gamma}_\nu(n,t)=\frac{(-1)^n}{n!}\sum_{k=0}^{\infty‎}\frac{C_k(-\lambda^\nu t^{-\beta})^k}{\Gamma(1-k\beta)}\frac{\Gamma(k\nu+1)}{\Gamma(k\nu+1-n)},\ n\geq0,
\end{equation}
where
\begin{equation}
C_k=\prod_{j=1}^k\frac{\Gamma(1+\gamma-j\beta)}{\Gamma(1+\gamma+\alpha-(j-1)\beta)}.
\end{equation}
\end{theorem}
\begin{proof}
Applying $I^{\alpha,\beta,\gamma}_t$ on both sides of (\ref{rnewq}) and using (\ref{2.478}), we obtain
\begin{equation}\label{kol7}
p^{\alpha,\beta,\gamma}_{\nu}(n,t)=p^{\alpha,\beta,\gamma}_{\nu}(n,0)-\lambda^\nu I^{\alpha,\beta,\gamma}_t \sum_{r=0}^n (-1)^r\frac{(\nu)_r}{r!}p^{\alpha,\beta,\gamma}_{\nu}(n-r,t),\ \ n\geq 0.
\end{equation}

For $n=0$, substituting $p^{\alpha,\beta,\gamma}_{\nu}(0,t)=\sum_{k=0}^{\infty}p^{\alpha,\beta,\gamma}_{\nu,k}(0,t)$ in (\ref{kol7}) and applying ADM (see (\ref{1.4})), we get
\begin{equation*}\label{2.7kk7}
\sum_{k=0}^{\infty}p^{\alpha,\beta,\gamma}_{\nu,k}(0,t)=p^{\alpha,\beta,\gamma}_{\nu}(0,0)-\lambda^\nu\sum_{k=0}^{\infty} I^{\alpha,\beta,\gamma}_t p^{\alpha,\beta,\gamma}_{\nu,k}(0,t).
\end{equation*}
Thus, $p^{\alpha,\beta,\gamma}_{\nu,0}(0,t)=p^{\alpha,\beta,\gamma}_{\nu}(0,0)=1$ and $p^{\alpha,\beta,\gamma}_{\nu,k}(0,t)=-\lambda^\nu I^{\alpha,\beta,\gamma}_t p^{\alpha,\beta,\gamma}_{\nu,k-1}(0,t)$, $k\geq 1$.\\
Hence,
\begin{align*}
p^{\alpha,\beta,\gamma}_{\nu,1}(0,t)=-\lambda^\nu I^{\alpha,\beta,\gamma}_t p^{\alpha,\beta,\gamma}_{\nu,0}(0,t)
&=-\lambda^\nu I^{\alpha,\beta,\gamma}_tt^0
=\frac{C_1(-\lambda^\nu t^{-\beta})}{\Gamma(1-\beta)},\\
p^{\alpha,\beta,\gamma}_{\nu,2}(0,t)=-\lambda^\nu I^{\alpha,\beta,\gamma}_t p^{\alpha,\beta,\gamma}_{\nu,1}(0,t)&=\frac{\lambda^{2\nu}C_1}{\Gamma(1-\beta)}I^{\alpha,\beta,\gamma}_t t^{-\beta}=\frac{C_2(-\lambda^\nu t^{-\beta})^2}{\Gamma(1-2\beta)},\\
p^{\alpha,\beta,\gamma}_{\nu,3}(0,t)=-\lambda^\nu I^{\alpha,\beta,\gamma}_t p^{\alpha,\beta,\gamma}_{\nu,2}(0,t)&=\frac{-\lambda^{3\nu}C_2}{\Gamma(1-2\beta)}I^{\alpha,\beta,\gamma}_t t^{-2\beta}=\frac{C_3(-\lambda^\nu t^{-\beta})^3}{\Gamma(1-3\beta)}.
\end{align*}
Let
\begin{equation}
p^{\alpha,\beta,\gamma}_{\nu,k-1}(0,t)=\frac{C_{k-1}(-\lambda^\nu t^{-\beta})^{k-1}}{\Gamma(1-(k-1)\beta)}.
\end{equation}
Then
\begin{equation*}
p^{\alpha,\beta,\gamma}_{\nu,k}(0,t)=-\lambda^\nu I^{\alpha,\beta,\gamma}_t p^{\alpha,\beta,\gamma}_{\nu,k-1}(0,t)
=\frac{(-\lambda^\nu)^kC_{k-1}}{\Gamma(1-(k-1)\beta)} I^{\alpha,\beta,\gamma}_t t^{-(k-1)\beta}=\frac{C_k(-\lambda^\nu t^{-\beta})^k}{\Gamma(1-k\beta)},\ \ k\geq 0.
\end{equation*}
Therefore
\begin{equation}
p^{\alpha,\beta,\gamma}_{\nu}(0,t)=\sum_{k=0}^{\infty}\frac{C_k(-\lambda^\nu t^{-\beta})^k}{\Gamma(1-k\beta)},
\end{equation}
{\it i.e.} the result holds for $n=0$.

For $n=1$, substituting $p^{\alpha,\beta,\gamma}_{\nu}(1,t)=\sum_{k=0}^{\infty}p^{\alpha,\beta,\gamma}_{\nu,k}(1,t)$ in (\ref{kol7}) and applying ADM, we get
\begin{equation*}\label{2.7kk7}
\sum_{k=0}^{\infty}p^{\alpha,\beta,\gamma}_{\nu,k}(1,t)=p^{\alpha,\beta,\gamma}_{\nu}(1,0)-\lambda^\nu\sum_{k=0}^{\infty} I^{\alpha,\beta,\gamma}_t \left(p^{\alpha,\beta,\gamma}_{\nu,k}(1,t)-\nu p^{\alpha,\beta,\gamma}_{\nu,k}(0,t)\right).
\end{equation*}
Thus, $p^{\alpha,\beta,\gamma}_{\nu,0}(1,t)=p^{\alpha,\beta,\gamma}_{\nu}(1,0)=0$ and 
\begin{equation*}
p^{\alpha,\beta,\gamma}_{\nu,k}(1,t)=-\lambda^\nu I^{\alpha,\beta,\gamma}_t  \left(p^{\alpha,\beta,\gamma}_{\nu,k-1}(1,t)-\nu p^{\alpha,\beta,\gamma}_{\nu,k-1}(0,t)\right),\ \ k\geq 1.
\end{equation*}
Hence,
\begin{align*}
p^{\alpha,\beta,\gamma}_{\nu,1}(1,t)&=-\lambda^\nu I^{\alpha,\beta,\gamma}_t  \left(p^{\alpha,\beta,\gamma}_{\nu,0}(1,t)-\nu p^{\alpha,\beta,\gamma}_{\nu,0}(0,t)\right)
=\frac{-\nu C_1(-\lambda^\nu t^{-\beta})}{\Gamma(1-\beta)},\\
p^{\alpha,\beta,\gamma}_{\nu,2}(1,t)&=-\lambda^\nu I^{\alpha,\beta,\gamma}_t  \left(p^{\alpha,\beta,\gamma}_{\nu,1}(1,t)-\nu p^{\alpha,\beta,\gamma}_{\nu,1}(0,t)\right)
=\frac{-2\nu C_2(-\lambda^\nu t^{-\beta})^2}{\Gamma(1-2\beta)},\\
p^{\alpha,\beta,\gamma}_{\nu,3}(1,t)&=-\lambda^\nu I^{\alpha,\beta,\gamma}_t  \left(p^{\alpha,\beta,\gamma}_{\nu,2}(1,t)-\nu p^{\alpha,\beta,\gamma}_{\nu,2}(0,t)\right)
=\frac{-3\nu C_3(-\lambda^\nu t^{-\beta})^3}{\Gamma(1-3\beta)}.
\end{align*}
Let
\begin{equation}
p^{\alpha,\beta,\gamma}_{\nu,k-1}(1,t)=\frac{-(k-1)\nu C_{k-1}(-\lambda^\nu t^{-\beta})^{k-1}}{\Gamma(1-(k-1)\beta)}.
\end{equation}
Then
\begin{equation*}
p^{\alpha,\beta,\gamma}_{\nu,k}(1,t)=-\lambda^\nu I^{\alpha,\beta,\gamma}_t  \left(p^{\alpha,\beta,\gamma}_{\nu,k-1}(1,t)-\nu p^{\alpha,\beta,\gamma}_{\nu,k-1}(0,t)\right)=\frac{-k\nu C_k(-\lambda^\nu t^{-\beta})^k}{\Gamma(1-k\beta)},\ \ k\geq 0.
\end{equation*}
Therefore
\begin{equation}
p^{\alpha,\beta,\gamma}_{\nu}(1,t)=-\sum_{k=0}^{\infty}\frac{k\nu C_k(-\lambda^\nu t^{-\beta})^k}{\Gamma(1-k\beta)},
\end{equation}
{\it i.e.} the result holds for $n=1$.

Now assume for $m>1$ the following:
\begin{equation}
p^{\alpha,\beta,\gamma}_{\nu,k}(m,t)=\frac{(-1)^m}{m!}\frac{(k\nu)_mC_k(-\lambda^\nu t^{-\beta})^k}{\Gamma(1-k\beta)},\ \ k\geq 0.
\end{equation}
{\it i.e.} (\ref{2.4kkr}) holds for $n=m$, where $p^{\alpha,\beta,\gamma}_{\nu}(m,t)=\sum_{k=0}^{\infty}p^{\alpha,\beta,\gamma}_{\nu,k}(m,t)$.

For $n=m+1$, substituting $p^{\alpha,\beta,\gamma}_{\nu}(m+1,t)=\sum_{k=0}^{\infty}p^{\alpha,\beta,\gamma}_{\nu,k}(m+1,t)$ in (\ref{kol7}) and applying ADM, we get
\begin{equation*}\label{2.7kk7}
\sum_{k=0}^{\infty}p^{\alpha,\beta,\gamma}_{\nu,k}(m+1,t)=p^{\alpha,\beta,\gamma}_{\nu}(m+1,0)-\lambda^\nu\sum_{k=0}^{\infty} I^{\alpha,\beta,\gamma}_t \sum_{r=0}^{m+1} (-1)^r\frac{(\nu)_r}{r!}p^{\alpha,\beta,\gamma}_{\nu,k}(m+1-r,t).
\end{equation*}
Thus, $p^{\alpha,\beta,\gamma}_{\nu,0}(m+1,t)=p^{\alpha,\beta,\gamma}_{\nu}(m+1,0)=0$ and
\begin{equation*}
p^{\alpha,\beta,\gamma}_{\nu,k}(m+1,t)=-\lambda^\nu I^{\alpha,\beta,\gamma}_t \sum_{r=0}^{m+1} (-1)^r\frac{(\nu)_r}{r!}p^{\alpha,\beta,\gamma}_{\nu,k-1}(m+1-r,t),\ \ k\geq 1.
\end{equation*}
Hence,
\begin{align*}
p^{\alpha,\beta,\gamma}_{\nu,1}(m+1,t)&=-\lambda^\nu I^{\alpha,\beta,\gamma}_t \sum_{r=0}^{m+1} (-1)^r\frac{(\nu)_r}{r!}p^{\alpha,\beta,\gamma}_{\nu,0}(m+1-r,t)\\
&=\frac{(-1)^{m+1}}{(m+1)!}\frac{(\nu)_{m+1}C_1(-\lambda^\nu t^{-\beta})}{\Gamma(1-\beta)},\\
p^{\alpha,\beta,\gamma}_{\nu,2}(m+1,t)&=-\lambda^\nu I^{\alpha,\beta,\gamma}_t \sum_{r=0}^{m+1} (-1)^r\frac{(\nu)_r}{r!}p^{\alpha,\beta,\gamma}_{\nu,1}(m+1-r,t)\\
&=\frac{\lambda^{2\nu}(-1)^{m+1}C_1}{(m+1)!\Gamma(1-\beta)}I^{\alpha,\beta,\gamma}_tt^{-\beta}\sum_{r=0}^{m+1} \frac{(m+1)!}{r!(m+1-r)!}(\nu)_r(\nu)_{m+1-r}\\
&=\frac{(-1)^{m+1}}{(m+1)!}\frac{(2\nu)_{m+1}C_2(-\lambda^\nu t^{-\beta})^2}{\Gamma(1-2\beta)}.
\end{align*}
Let
\begin{equation*}
p^{\alpha,\beta,\gamma}_{\nu,k-1}(m+1,t)=\frac{(-1)^{m+1}}{(m+1)!}\frac{((k-1)\nu)_{m+1}C_{k-1}(-\lambda^\nu t^{-\beta})^{k-1}}{\Gamma(1-(k-1)\beta)}.
\end{equation*}
Then
\begin{align*}
p^{\alpha,\beta,\gamma}_{\nu,k}(m+1,t)&=-\lambda^\nu I^{\alpha,\beta,\gamma}_t \sum_{r=0}^{m+1} (-1)^r\frac{(\nu)_r}{r!}p^{\alpha,\beta,\gamma}_{\nu,k-1}(m+1-r,t)\\
&=\frac{(-\lambda^\nu)^k(-1)^{m+1}C_{k-1}I^{\alpha,\beta,\gamma}_tt^{-(k-1)\beta}}{(m+1)!\Gamma(1-(k-1)\beta)}\sum_{r=0}^{m+1} \frac{(m+1)!}{r!(m+1-r)!}(\nu)_r((k-1)\nu)_{m+1-r}\\
&=\frac{(-1)^{m+1}}{(m+1)!}\frac{(k\nu)_{m+1}C_k(-\lambda^\nu t^{-\beta})^k}{\Gamma(1-k\beta)},\ \ k\geq 0.
\end{align*}
Therefore
\begin{equation*}
p^{\alpha,\beta,\gamma}_{\nu}(m+1,t)=\frac{(-1)^{m+1}}{(m+1)!}\sum_{k=0}^{\infty‎}\frac{C_k(-\lambda^\nu t^{-\beta})^k}{\Gamma(1-k\beta)}\frac{\Gamma(k\nu+1)}{\Gamma(k\nu-m)},
\end{equation*}
and thus the result holds for $n=m+1$. This completes the proof.
\end{proof}
Next we show that $p^{\alpha,\beta,\gamma}_{\nu}(n,t)$ is indeed a pmf. Note that
\begin{align*}
\sum_{n=0}^{\infty‎}p^{\alpha,\beta,\gamma}_\nu(n,t)&=\sum_{n=0}^{\infty‎}\frac{(-1)^n}{n!}\sum_{k=0}^{\infty‎}\frac{C_k(-\lambda^\nu t^{-\beta})^k}{\Gamma(1-k\beta)}\frac{\Gamma(k\nu+1)}{\Gamma(k\nu+1-n)}\\
&=\sum_{k=0}^{\infty‎}\frac{C_k(-\lambda^\nu t^{-\beta})^k}{\Gamma(1-k\beta)}\sum_{n=0}^{\infty‎}(k\nu)_n\frac{(-1)^n}{n!}\\
&=\sum_{k=0}^{\infty‎}\frac{C_k(-\lambda^\nu t^{-\beta})^k}{\Gamma(1-k\beta)}(1-1)^{k\nu}=1,
\end{align*}
since for $k\nu\geq 0$ the binomial series $\sum_{n=0}^{\infty‎}(k\nu)_n\frac{(-1)^n}{n!}$ converges absolutely and also all the terms except for $k=0$ vanishes. For $\beta=-\alpha$, the pmf $p^{\alpha,\beta,\gamma}_{\nu}(n,t)$ reduces to that of the STFPP.  
\begin{corollary}
Let the random variable $X^{\alpha,\beta,\gamma}_\nu$ be the waiting time of the first Saigo space and time fractional Poisson event. Then
\begin{equation*}
\mathrm{Pr}\{X^{\alpha,\beta,\gamma}_\nu>t\}=\mathrm{Pr}\{N^{\alpha,\beta,\gamma}_\nu(t,\lambda)=0\}=\sum_{k=0}^{\infty‎}\frac{C_k(-\lambda^\nu t^{-\beta})^k}{\Gamma(1-k\beta)},\ \ t\geq 0.
\end{equation*}
The special case $\beta=-\alpha$ corresponds to Mittag-Leffler distribution (\ref{77}) {\it i.e.} the first waiting time of STFPP.
\end{corollary}
\begin{theorem}\label{t2kekt}
The probability generating function $G^{\alpha,\beta,\gamma}_\nu(u,t)=\mathbb{E}(u^{N^{\alpha,\beta,\gamma}_\nu(t,\lambda)})$, of the SSTFPP is
\begin{equation}\label{2.4rker}
G^{\alpha,\beta,\gamma}_\nu(u,t)=\sum_{k=0}^{\infty‎}\frac{C_k(-\lambda^\nu (1-u)^\nu t^{-\beta})^k}{\Gamma(1-k\beta)},\ \ |u|<1.
\end{equation}
\end{theorem}
\begin{proof}
We have
\begin{align*}
G^{\alpha,\beta,\gamma}_\nu(u,t)&=\sum_{n=0}^\infty u^np^{\alpha,\beta,\gamma}_\nu(n,t)\\
&=\sum_{k=0}^{\infty‎}\frac{C_k(-\lambda^\nu t^{-\beta})^k}{\Gamma(1-k\beta)}\sum_{n=0}^\infty\frac{\Gamma(k\nu+1)}{\Gamma(k\nu+1-n)}\frac{(-u)^n}{n!},
\end{align*}
and thus the proof follows on using the generalized binomial theorem.
\end{proof}
\begin{corollary}
The probability generating function of the SSTFPP satisfies the following Cauchy Problem:
\begin{align*}
\partial_{t}^{\alpha,\beta,\gamma}G^{\alpha,\beta,\gamma}_\nu(u,t)&=-\lambda^\nu G^{\alpha,\beta,\gamma}_\nu(u,t)(1-u)^\nu,\ \ |u|<1,\\
G^{\alpha,\beta,\gamma}_\nu(u,0)&=1.
\end{align*}
\end{corollary}
\begin{corollary}
The probability generating functions of STFPP ($\beta=-\alpha$), SFPP ($\beta=-\alpha=-1$), and TFPP ($\beta=-\alpha$, $\nu=1$) are
\begin{align*}
G^{\alpha}_\nu(u,t)&=E_{\alpha}(-\lambda^\nu(1-u)^\nu t^\alpha),\ \ |u|<1,\\
G_\nu(u,t)&=e^{-\lambda^\nu(1-u)^\nu t},\ \ |u|<1,\\
G^{\alpha}(u,t)&=E_{\alpha}(-\lambda(1-u)t^\alpha),\ \ |u|<1,
\end{align*}
respectively.
\end{corollary}

\section{Concluding remarks}
The state probabilities of various fractional generalizations of the classical homogeneous Poisson process are obtained by several authors by evaluating probability generating function using Laplace transform. Sometimes Laplace transform of certain fractional derivatives may not be known or may have complicated forms. In this paper, we have shown that ADM can be effectively used to obtain these state probabilities. As an illustration, we have obtained the distribution of STFPP using ADM. We have also improved a result of {\small{\sc Rao}} {\it et al.} (2010) by introducing the correct version of the Caputo type Saigo fractional derivative. We used Caputo type Saigo fractional derivative to generalize STFPP to SSTFPP and the state probabilities of SSTFPP are obtained using ADM.

\newpage
\section*{Supplementary}
Here we have illustrated the use of ADM to obtain the state probabilities of TFPP and SFPP.
\subsection*{Time fractional Poisson process}
Consider the following difference-differential equations governing the state probabilities of the TFPP:
\begin{equation}\label{r2new}
\partial_t^\alpha p^\alpha(n,t)=-\lambda(p^\alpha(n,t)-p^\alpha(n-1,t)),\ \ 0<\alpha\leq 1,\ n\geq 0,
\end{equation}
with $p^\alpha(0,0)=1$ and $p^\alpha(n,0)=0$, $n\geq 1$. The solution of (\ref{r2new}) is given by
\begin{equation}\label{2.4kky}
p^\alpha(n,t)=\frac{(\lambda t^\alpha)^n}{n!}\sum_{k=0}^{\infty‎}\frac{(k+n)!}{k!}\frac{(-\lambda t^\alpha)^k}{\Gamma\left((k+n)\alpha+1\right)},\ \ n\geq0.
\end{equation}

\begin{proof}
Applying RL integral $I^\alpha_t$ on both sides of (\ref{r2new}), we get
\begin{equation}\label{yth}
p^\alpha(n,t)=p^\alpha(n,0)-\lambda I_t^\alpha(p^\alpha(n,t)-p^\alpha(n-1,t)),\ \ n\geq 0.
\end{equation}

Note that $p^\alpha(-1,t)=0$ for $t\geq0$. For $n=0$, Substitute $p^\alpha(0,t)=\sum_{k=0}^{\infty}p^\alpha_{k}(0,t)$ in (\ref{yth}) and apply ADM to get
\begin{equation*}\label{2.7kk}
\sum_{k=0}^{\infty}p^\alpha_{k}(0,t)=p^\alpha(0,0)-\lambda\sum_{k=0}^{\infty} I_t^\alpha p^\alpha_{k}(0,t).
\end{equation*}
Thus, $p^\alpha_{0}(0,t)=p^\alpha(0,0)=1$ and $p^\alpha_{k}(0,t)=-\lambda I_t^\alpha p^\alpha_{k-1}(0,t)$, $k\geq 1$. Hence,
\begin{equation*}
p^\alpha_{1}(0,t)=-\lambda I_t^\alpha p^\alpha_{0}(0,t)=-\lambda I_t^\alpha t^0=\frac{-\lambda t^\alpha}{\Gamma(\alpha+1)},
\end{equation*}
and similarly 
\begin{equation*}
p^\alpha_{2}(0,t)=\frac{(-\lambda t^{\alpha})^2}{\Gamma(2\alpha+1)},\ \ p^\alpha_{3}(0,t)=\frac{(-\lambda t^{\alpha})^3}{\Gamma(3\alpha+1)},\ldots.
\end{equation*}
Let
\begin{equation}
p^\alpha_{k-1}(0,t)=\frac{(-\lambda t^{\alpha})^{k-1}}{\Gamma((k-1)\alpha+1)}.
\end{equation}
Then
\begin{equation*}
p^\alpha_{k}(0,t)=-\lambda I_t^\alpha p^\alpha_{k-1}(0,t)=\frac{(-\lambda)^{k}}{\Gamma((k-1)\alpha+1)} I_t^\alpha t^{(k-1)\alpha}=\frac{(-\lambda t^{\alpha})^{k}}{\Gamma(k\alpha+1)},\ \ k\geq 0.
\end{equation*}
Therefore
\begin{equation}
p^\alpha(0,t)=\sum_{k=0}^{\infty}\frac{(-\lambda t^{\alpha})^{k}}{\Gamma(k\alpha+1)},
\end{equation}
and thus the result holds for $n=0$.
 
For $n=1$, substituting $p^\alpha(1,t)=\sum_{k=0}^{\infty}p^\alpha_{k}(1,t)$ in (\ref{yth}) and applying ADM, we get
\begin{equation*}\label{2.7kk}
\sum_{k=0}^{\infty}p^\alpha_{k}(1,t)=p^\alpha(1,0)-\lambda\sum_{k=0}^{\infty} I_t^\alpha \left(p^\alpha_{k}(1,t)-p^\alpha_{k}(0,t)\right).
\end{equation*}
Thus, $p^\alpha_{0}(1,t)=p^\alpha(1,0)=0$ and $p^\alpha_{k}(1,t)=-\lambda I_t^\alpha \left(p^\alpha_{k-1}(1,t)-p^\alpha_{k-1}(0,t)\right)$, $k\geq 1$.\\
Hence,
\begin{align*}
p^\alpha_{1}(1,t)&=-\lambda I_t^\alpha \left(p^\alpha_{0}(1,t)-p^\alpha_{0}(0,t)\right)=\lambda I_t^\alpha t^0=\frac{-(-\lambda t^\alpha)}{\Gamma(\alpha+1)},\\
p^\alpha_{2}(1,t)&=-\lambda I_t^\alpha \left(p^\alpha_{1}(1,t)-p^\alpha_{1}(0,t)\right)=\frac{-2\lambda^2}{\Gamma(\alpha+1)} I_t^\alpha t^\alpha=\frac{-2(-\lambda t^{\alpha})^2}{\Gamma(2\alpha+1)},\\
p^\alpha_{3}(1,t)&=-\lambda I_t^\alpha \left(p^\alpha_{2}(1,t)-p^\alpha_{2}(0,t)\right)=\frac{3\lambda^3}{\Gamma(2\alpha+1)} I_t^\alpha t^{2\alpha}=\frac{-3(-\lambda t^{\alpha})^3}{\Gamma(3\alpha+1)}.
\end{align*}
Let
\begin{equation}
p^\alpha_{k-1}(1,t)=\frac{-(k-1)(-\lambda t^{\alpha})^{k-1}}{\Gamma((k-1)\alpha+1)}.
\end{equation}
Then
\begin{align*}
p^\alpha_{k}(1,t)=-\lambda I_t^\alpha \left(p^\alpha_{k-1}(1,t)-p^\alpha_{k-1}(0,t)\right)&=\frac{(-1)^{k+1}k\lambda^{k}}{\Gamma((k-1)\alpha+1)} I_t^\alpha t^{(k-1)\alpha}\\&=\frac{-k(-\lambda t^{\alpha})^{k}}{\Gamma(k\alpha+1)},\ \ k\geq 1.
\end{align*}
Therefore
\begin{equation}
p^\alpha(1,t)=-\sum_{k=1}^{\infty}\frac{k(-\lambda t^{\alpha})^{k}}{\Gamma(k\alpha+1)}=\lambda t^{\alpha}\sum_{k=0}^{\infty}\frac{(k+1)(-\lambda t^{\alpha})^{k}}{\Gamma((k+1)\alpha+1)},
\end{equation}
and thus the result holds for $n=1$.

For $n=2$, substituting $p^\alpha(2,t)=\sum_{k=0}^{\infty}p^\alpha_{k}(2,t)$ in (\ref{yth}) and applying ADM, we get
\begin{equation*}\label{2.7kk}
\sum_{k=0}^{\infty}p^\alpha_{k}(2,t)=p^\alpha(2,0)-\lambda\sum_{k=0}^{\infty} I_t^\alpha \left(p^\alpha_{k}(2,t)-p^\alpha_{k}(1,t)\right).
\end{equation*}
Thus, $p^\alpha_{0}(2,t)=p^\alpha(2,0)=0$ and $p^\alpha_{k}(2,t)=-\lambda I_t^\alpha \left(p^\alpha_{k-1}(2,t)-p^\alpha_{k-1}(1,t)\right)$, $k\geq 1$.\\
Hence,
\begin{align*}
p^\alpha_{1}(2,t)&=-\lambda I_t^\alpha \left(p^\alpha_{0}(2,t)-p^\alpha_{0}(1,t)\right)=0,\\
p^\alpha_{2}(2,t)&=-\lambda I_t^\alpha \left(p^\alpha_{1}(2,t)-p^\alpha_{1}(1,t)\right)=\frac{\lambda^2}{\Gamma(\alpha+1)} I_t^\alpha t^\alpha=\frac{2.1(-\lambda t^{\alpha})^2}{2\Gamma(2\alpha+1)},\\
p^\alpha_{3}(2,t)&=-\lambda I_t^\alpha \left(p^\alpha_{2}(2,t)-p^\alpha_{2}(1,t)\right)=\frac{-3\lambda^3}{\Gamma(2\alpha+1)} I_t^\alpha t^{2\alpha}=\frac{3.2(-\lambda t^{\alpha})^3}{2\Gamma(3\alpha+1)}.
\end{align*}
Let
\begin{equation}
p^\alpha_{k-1}(2,t)=\frac{(k-1)(k-2)(-\lambda t^{\alpha})^{k-1}}{2\Gamma((k-1)\alpha+1)}.
\end{equation}
Then
\begin{align*}
p^\alpha_{k}(2,t)&=-\lambda I_t^\alpha \left(p^\alpha_{k-1}(2,t)-p^\alpha_{k-1}(1,t)\right)\\
&=\frac{(-1)^{k}k(k-1)\lambda^{k}}{2\Gamma((k-1)\alpha+1)} I_t^\alpha t^{(k-1)\alpha}=\frac{k(k-1)(-\lambda t^{\alpha})^{k}}{2\Gamma(k\alpha+1)},\ \ k\geq 2.
\end{align*}
Therefore
\begin{equation}
p^\alpha(2,t)=\sum_{k=2}^{\infty}\frac{k(k-1)(-\lambda t^{\alpha})^{k}}{2\Gamma(k\alpha+1)}=\frac{(\lambda t^{\alpha})^2}{2}\sum_{k=0}^{\infty}\frac{(k+2)(k+1)(-\lambda t^{\alpha})^{k}}{\Gamma((k+2)\alpha+1)},
\end{equation}
and thus the result holds for $n=2$.

Let $p^\alpha(m,t)=\sum_{k=0}^{\infty}p^\alpha_{k}(m,t)$ in (\ref{yth}) and assume the result holds for $n=m>2$ {\it i.e.} $p^\alpha_{k}(m,t)=0$, $k<m$ and
\begin{equation*}
p^\alpha_{k}(m,t)=\frac{(-1)^{m}k!(-\lambda t^{\alpha})^{k}}{m!(k-m)!\Gamma(k\alpha+1)},\ \ k\geq m.
\end{equation*}
For $n=m+1$, substituting $p^\alpha(m+1,t)=\sum_{k=0}^{\infty}p^\alpha_{k}(m+1,t)$ in (\ref{yth}) and applying ADM, we get
\begin{equation*}
\sum_{k=0}^{\infty}p^\alpha_{k}(m+1,t)=p^\alpha(m+1,0)-\lambda\sum_{k=0}^{\infty} I_t^\alpha \left(p^\alpha_{k}(m+1,t)-p^\alpha_{k}(m,t)\right).
\end{equation*}
Thus, $p^\alpha_{0}(m+1,t)=p^\alpha(m+1,0)=0$ and $p^\alpha_{k}(m+1,t)=-\lambda I_t^\alpha \left(p^\alpha_{k-1}(m+1,t)-p^\alpha_{k-1}(m,t)\right)$, $k\geq 1$. Hence,
\begin{align*}
p^\alpha_{1}(m+1,t)&=-\lambda I_t^\alpha \left(p^\alpha_{0}(m+1,t)-p^\alpha_{0}(m,t)\right)=0,\\
p^\alpha_{2}(m+1,t)&=-\lambda I_t^\alpha \left(p^\alpha_{1}(m+1,t)-p^\alpha_{1}(m,t)\right)=0.
\end{align*}
Let
\begin{equation*}
p^\alpha_{k-1}(m+1,t)=0,\ \ k-1<m+1.
\end{equation*}
Then
\begin{equation*}
p^\alpha_{k}(m+1,t)=-\lambda I_t^\alpha \left(p^\alpha_{k-1}(m+1,t)-p^\alpha_{k-1}(m,t)\right)=0,\ \ k<m+1.
\end{equation*}
Now for $k\geq m+1$, we have
\begin{align*}
p^\alpha_{m+1}(m+1,t)&=-\lambda I_t^\alpha \left(p^\alpha_{m}(m+1,t)-p^\alpha_{m}(m,t)\right)=\frac{\lambda^{m+1}}{\Gamma(m\alpha+1)} I_t^\alpha t^{m\alpha}=\frac{(\lambda t^{\alpha})^{m+1}}{\Gamma((m+1)\alpha+1)},\\
p^\alpha_{m+2}(m+1,t)&=-\lambda I_t^\alpha \left(p^\alpha_{m+1}(m+1,t)-p^\alpha_{m+1}(m,t)\right)\\
&=\frac{-(m+2)\lambda^{m+2}}{\Gamma((m+1)\alpha+1)} I_t^\alpha t^{(m+1)\alpha}=\frac{-(m+2)(\lambda t^{\alpha})^{m+2}}{\Gamma((m+2)\alpha+1)}.
\end{align*}
Let
\begin{equation*}
p^\alpha_{k-1}(m+1,t)=\frac{(-1)^{m+1}(k-1)!(-\lambda t^{\alpha})^{k-1}}{(m+1)!(k-m-2)!\Gamma((k-1)\alpha+1)},\ \ k-1\geq m+1.
\end{equation*}
Then
\begin{align*}
p^\alpha_{k}(m+1,t)&=-\lambda I_t^\alpha \left(p^\alpha_{k-1}(m+1,t)-p^\alpha_{k-1}(m,t)\right)\\
&=\frac{(-1)^{k+m+1}k!\lambda^{k}}{(m+1)!(k-m-1)!\Gamma((k-1)\alpha+1)} I_t^\alpha t^{(k-1)\alpha}\\
&=\frac{(-1)^{k+m+1}k!(\lambda t^{\alpha})^{k}}{(m+1)!(k-m-1)!\Gamma(k\alpha+1)},\ \ k\geq m+1.
\end{align*}
Therefore
\begin{align*}
p^\alpha(m+1,t)&=\sum_{k=m+1}^{\infty}\frac{(-1)^{k+m+1}k!(\lambda t^{\alpha})^{k}}{(m+1)!(k-m-1)!\Gamma(k\alpha+1)}\\
&=\frac{(\lambda t^{\alpha})^{m+1}}{(m+1)!}\sum_{k=0}^{\infty}\frac{(k+m)!(-\lambda t^{\alpha})^{k}}{k!\Gamma((k+m+1)\alpha+1)},
\end{align*}
and thus the result holds for $n=m+1$. This completes the proof.
\end{proof}
\subsection*{Space fractional Poisson process}
Consider the following difference-differential equations governing the state probabilities of the SFPP:
\begin{equation}\label{hj45}
\frac{\mathrm{d}}{\mathrm{d}t} p_\nu(n,t)=-\lambda^\nu\sum_{r=0}^{n}(-1)^r\frac{(\nu)_r}{r!} p_\nu(n-r,t),\ \ 0<\nu\leq 1,\ n\geq 0,
\end{equation}
with $p_\nu(0,0)=1$ and $p_\nu(n,0)=0$, $n\geq 1$. The solution of (\ref{hj45}) is given by
\begin{equation}\label{2.4kk}
p_\nu(n,t)=\frac{(-1)^n}{n!}\sum_{k=0}^{\infty‎}\frac{(-\lambda^\nu t)^k}{k!}\frac{\Gamma(k\nu+1)}{\Gamma(k\nu+1-n)},\ \ n\geq0.
\end{equation}
\begin{proof}
The difference-differential equations (\ref{hj45}) can be equivalently written as
\begin{equation}\label{rgft}
p_\nu(n,t)=p_\nu(n,0)-\lambda^\nu\int_0^t \sum_{r=0}^n (-1)^r\frac{(\nu)_r}{r!}p_\nu(n-r,s)\,\mathrm{d}s,\ \ n\geq 0.
\end{equation}

For $n=0$, substituting $p_\nu(0,t)=\sum_{k=0}^{\infty}p_{\nu,k}(0,t)$ in (\ref{rgft}) and applying ADM, we get
\begin{equation*}\label{2.7kk}
\sum_{k=0}^{\infty}p_{\nu,k}(0,t)=p_\nu(0,0)-\lambda^\nu\sum_{k=0}^{\infty} \int_0^t  p_{\nu,k}(0,s)\,\mathrm{d}s.
\end{equation*}
Thus, $p_{\nu,0}(0,t)=p_\nu(0,0)=1$ and $p_{\nu,k}(0,t)=-\lambda^\nu \int_0^t p_{\nu,k-1}(0,s)\,\mathrm{d}s$, $k\geq 1$.\\
Hence,
\begin{align*}
p_{\nu,1}(0,t)&=-\lambda^\nu \int_0^t p_{\nu,0}(0,s)\,\mathrm{d}s
=-\lambda^\nu \int_0^t\,\mathrm{d}s
=-\lambda^\nu t,\\
p_{\nu,2}(0,t)&=-\lambda^\nu \int_0^t p_{\nu,1}(0,s)\,\mathrm{d}s
=-\lambda^\nu \int_0^t (-\lambda^\nu s)\,\mathrm{d}s
=\frac{(-\lambda^{\nu}t)^2}{2!},\\
p_{\nu,3}(0,t)&=-\lambda^\nu \int_0^t p_{\nu,2}(0,s)\,\mathrm{d}s
=-\lambda^\nu \int_0^t \frac{(\lambda^{\nu}s)^2}{2!}\,\mathrm{d}s
=\frac{(-\lambda^{\nu}t)^3}{3!},\\
&\ \vdots\\
p_{\nu,k}(0,t)&=-\lambda^\nu \int_0^t p_{\nu,k-1}(0,s)\,\mathrm{d}s
=-\lambda^\nu \int_0^t \frac{(-\lambda^{\nu}s)^{k-1}}{(k-1)!}\,\mathrm{d}s
=\frac{(-\lambda^{\nu}t)^k}{k!},\ \ k\geq 1.
\end{align*}
Therefore
\begin{equation}
p_{\nu}(0,t)=\sum_{k=0}^{\infty}\frac{(-\lambda^{\nu}t)^k}{k!}.
\end{equation}

For $n=1$, substituting $p_\nu(1,t)=\sum_{k=0}^{\infty}p_{\nu,k}(1,t)$ in (\ref{rgft}) and applying ADM, we get
\begin{equation*}\label{2.7kk}
\sum_{k=0}^{\infty}p_{\nu,k}(1,t)=p_\nu(1,0)-\lambda^\nu\sum_{k=0}^{\infty} \int_0^t \left(p_{\nu,k}(1,s)-\nu p_{\nu,k}(0,s)\right)\,\mathrm{d}s.
\end{equation*}
Thus, $p_{\nu,0}(1,t)=p_\nu(1,0)=0$ and $p_{\nu,k}(1,t)=-\lambda^\nu \int_0^t  \left(p_{\nu,k-1}(1,s)-\nu p_{\nu,k-1}(0,s)\right)\,\mathrm{d}s$, $k\geq 1$. Hence,
\begin{align*}
p_{\nu,1}(1,t)&=-\lambda^\nu \int_0^t  \left(p_{\nu,0}(1,s)-\nu p_{\nu,0}(0,s)\right)\,\mathrm{d}s=-\nu(-\lambda^\nu t),\\
p_{\nu,2}(1,t)&=-\lambda^\nu \int_0^t  \left(p_{\nu,1}(1,s)-\nu p_{\nu,1}(0,s)\right)\,\mathrm{d}s=-\nu(-\lambda^\nu t)^2,\\
p_{\nu,3}(1,t)&=-\lambda^\nu \int_0^t  \left(p_{\nu,2}(1,s)-\nu p_{\nu,2}(0,s)\right)\,\mathrm{d}s=\frac{-\nu(-\lambda^\nu t)^3}{2!},\\
&\ \vdots\\
p_{\nu,k}(1,t)&=-\lambda^\nu \int_0^t  \left(p_{\nu,k-1}(1,s)-\nu p_{\nu,k-1}(0,s)\right)\,\mathrm{d}s\\
&=-\lambda^\nu \int_0^t  \left(-\frac{\nu(-\lambda^\nu s)^{k-1}}{(k-2)!}-\frac{\nu(-\lambda^{\nu}s)^{k-1}}{(k-1)!}\right)\,\mathrm{d}s=\frac{-\nu(-\lambda^\nu t)^k}{(k-1)!},\ \ k\geq 1.
\end{align*}
Therefore
\begin{equation}
p_{\nu}(1,t)=-\sum_{k=1}^{\infty}\frac{\nu(-\lambda^\nu t)^k}{(k-1)!}=-\sum_{k=0}^{\infty}\frac{k\nu(-\lambda^\nu t)^k}{k!}.
\end{equation}

Now assume for $m>1$ the following:
\begin{equation*}
p_{\nu,k}(m,t)=\frac{(-1)^m}{m!}\frac{(k\nu)_m(-\lambda^\nu t)^k}{k!},\ \ k\geq 0,
\end{equation*}
{\it i.e.} (\ref{2.4kk}) holds for $n=m$, where $p_{\nu}(m,t)=\sum_{k=0}^{\infty}p_{\nu,k}(m,t)$.

For $n=m+1$, substituting $p_\nu(m+1,t)=\sum_{k=0}^{\infty}p_{\nu,k}(m+1,t)$ in (\ref{rgft}) and applying ADM, we get
\begin{equation*}\label{2.7kk}
\sum_{k=0}^{\infty}p_{\nu,k}(m+1,t)=p_\nu(m+1,0)-\lambda^\nu\sum_{k=0}^{\infty} \int_0^t \sum_{r=0}^{m+1} (-1)^r\frac{(\nu)_r}{r!}p_{\nu,k}(m+1-r,s)\,\mathrm{d}s.
\end{equation*}
Thus, $p_{\nu,0}(m+1,t)=p_\nu(m+1,0)=0$ and
\begin{equation*}
p_{\nu,k}(m+1,t)=-\lambda^\nu\int_0^t \sum_{r=0}^{m+1} (-1)^r\frac{(\nu)_r}{r!}p_{\nu,k-1}(m+1-r,s)\,\mathrm{d}s,\ \ k\geq 1.
\end{equation*}
Hence,
\begin{align*}
p_{\nu,1}(m+1,t)&=-\lambda^\nu\int_0^t \sum_{r=0}^{m+1} (-1)^r\frac{(\nu)_r}{r!}p_{\nu,0}(m+1-r,s)\,\mathrm{d}s\\
&=-\lambda^\nu\frac{(-1)^{m+1}}{(m+1)!}(\nu)_{m+1}\int_0^t\,\mathrm{d}s=\frac{(-1)^{m+1}}{(m+1)!}(\nu)_{m+1}(-\lambda^\nu t),\\
p_{\nu,2}(m+1,t)&=-\lambda^\nu\int_0^t \sum_{r=0}^{m+1} (-1)^r\frac{(\nu)_r}{r!}p_{\nu,1}(m+1-r,s)\,\mathrm{d}s\\
&=\frac{\lambda^{2\nu}(-1)^{m+1}}{(m+1)!}\int_0^ts\,\mathrm{d}s\sum_{r=0}^{m+1} \frac{(m+1)!}{r!(m+1-r)!}(\nu)_r(\nu)_{m+1-r}\\
&=\frac{(-1)^{m+1}}{(m+1)!}\frac{(2\nu)_{m+1}(-\lambda^\nu t)^2}{2!},
\end{align*}
where the last step follows from the binomial theorem for falling factorials. Now let
\begin{equation*}
p_{\nu,k-1}(m+1,t)=\frac{(-1)^{m+1}}{(m+1)!}\frac{((k-1)\nu)_{m+1}(-\lambda^\nu t)^{k-1}}{(k-1)!}.
\end{equation*}
Then
\begin{align*}
p_{\nu,k}(m+1,t)&=-\lambda^\nu\int_0^t \sum_{r=0}^{m+1} (-1)^r\frac{(\nu)_r}{r!}p_{\nu,k-1}(m+1-r,s)\,\mathrm{d}s\\
&=\frac{(-\lambda^\nu)^k(-1)^{m+1}}{(m+1)!(k-1)!}\int_0^ts^{k-1}\,\mathrm{d}s\sum_{r=0}^{m+1} \frac{(m+1)!}{r!(m+1-r)!}(\nu)_r((k-1)\nu)_{m+1-r}\\
&=\frac{(-1)^{m+1}}{(m+1)!}\frac{(k\nu)_{m+1}(-\lambda^\nu t)^k}{k!},\ \ k\geq 0.
\end{align*}
Therefore
\begin{equation*}
p^\alpha(m+1,t)=\frac{(-1)^{m+1}}{(m+1)!}\sum_{k=0}^{\infty‎}\frac{(-\lambda^\nu t)^k}{k!}\frac{\Gamma(k\nu+1)}{\Gamma(k\nu-m)},
\end{equation*}
and thus the result holds for $n=m+1$. This completes the proof.
\end{proof}

\begin{thebibliography}{1}



\bibitem{Adomian1986}
{\small {\sc G. Adomian:} {\it Nonlinear Stochastic Operator Equations.} Academic Press, Orlando, 1986.}
 
\bibitem{Adomian1994}
{\small {\sc G. Adomian:} {\it Solving Frontier Problems of Physics$:$ The Decomposition Method.} Kluwer Academic, Dordrecht, 1994.}

\bibitem{Beghin}
{\small {\sc L. Beghin, E. Orsingher:} {\it Fractional Poisson processes and related planar random motions.} Electron. J. Probab., {\bf 14} (2009), 1790--1826.}

\bibitem{Duan20101235}
{\small {\sc J.-S. Duan:} {\it Recurrence triangle for Adomian polynomials.} Appl. Math. Comput., {\bf 216} (2010), 1235--1241.}
 
\bibitem{Duan20116337}
{\small {\sc J.-S. Duan:} {\it Convenient analytic recurrence algorithms for the Adomian polynomials.} Appl. Math. Comput., {\bf 217} (2011), 6337--6348.}
 
\bibitem{Kataria2016}
{\small {\sc K. K. Kataria, P. Vellaisamy:} {\it Simple parametrization methods for generating Adomian polynomials.}
 Appl. Anal. Discrete Math., {\bf 10} (2016), 168--185.}
 
\bibitem{Kilbas2006}
{\small {\sc A. A. Kilbas, H. M. Srivastava, J. J. Trujillo:} {\it Theory and Applications of Fractional Differential Equations.} Elsevier, North Holland, 2006.}

\bibitem{Laskin}
{\small {\sc N. Laskin:} {\it Fractional Poisson Process.} Commun. Nonlinear Sci. Numer. Simul., {\bf 8} (2003), 201--213.}

\bibitem{Mark}
{\small {\sc M. M. Meerschaert, E. Nane, P. Vellaisamy:} {\it The fractional Poisson process and the inverse stable subordinator.} Electron. J. Probab., {\bf 16} (2011), 1600--1620.}

\bibitem{Orsingher}
{\small {\sc E. Orsingher, F. Polito:} {\it The space-fractional Poisson process.} Statist. Probab. Lett., {\bf 82} (2012), 852--858.}

\bibitem{Polito}
{\small {\sc F. Polito, E. Scalas:} {\it A generalization of the space-fractional Poisson process and its connection to some L\'evy processes.} Electron. Commun. Probab., {\bf 21} (2016), 1--14.}

\bibitem{Rach1984415}
{\small {\sc R. Rach:} {\it A convenient computational form for the Adomian polynomials.} J. Math. Anal. Appl., {\bf 102} (1984), 415--419.}

\bibitem{Rao15}
{\small {\sc A. Rao, M. Garg, S. L. Kalla:} {\it Caputo-type fractional derivative of a hypergeometric integral operator.} Kuwait J. Sci. Engrg., {\bf 37} (2010), 15--29.}

\bibitem{Saigo1978135}
{\small {\sc M. Saigo:} {\it A remark on integral operators involving the Gauss hypergeometric functions.} Math. Rep. Kyushu Univ., {\bf 11} (1978), 135--143.}

\bibitem{Saigo1998}
{\small {\sc M. Saigo, N. Maeda:} {\it More generalization of fractional calculus}, {Transform Methods \& Special Functions.} pp. 386--400, Varna 96 Proc. 2nd Intern. Workshop, Bulgaria Acad. Sci., Sofia, 1998.}

\bibitem{Srivastava1988}
{\small {\sc H. M. Srivastava, M. Saigo, S. Owa:} {\it A class of distortion theorems involving certain operators of fractional calculus.}  J. Math. Anal. Appl., {\bf 131} (1988), 412--420.}
\end{thebibliography}
\end{document}